\documentclass[12pt,twoside]{amsart}
\input amssymb.sty

\allowdisplaybreaks[4] \footskip=15pt
\renewcommand{\uppercasenonmath}[1]{}

\numberwithin{equation}{section} \theoremstyle{plain}
\newtheorem*{thm*}{Main Theorem}
\newtheorem{thm}{Theorem}[section]

\newtheorem*{cor*}{Corollary}
\newtheorem{lem}[thm]{Lemma}
\newtheorem*{lem*}{Lemma}
\newtheorem{prop}[thm]{Proposition}
\newtheorem*{prop*}{Proposition}

\newtheorem*{que*}{Question}
\newtheorem{rem}[thm]{Remark}
\newtheorem*{rem*}{Remark}
\newtheorem{exa}[thm]{Example}
\newtheorem*{exa*}{Example}
\newtheorem{df}[thm]{Definition}
\newtheorem*{df*}{Definition}

\newtheorem*{conj*}{Conjecture}
\newtheorem*{ack*}{ACKNOWLEDGEMENTS}

\begin{document}
\begin{center}
{\large  \bf ON COUNTABLY $\Sigma$-C2 RINGS}

\vspace{0.8cm} {\bf Liang Shen  and  Jianlong Chen\\
 {\it Department of Mathematics, Southeast University
 \\ Nanjing 210096, P.R. China}\\}

 {\it E-mail: lshen@seu.edu.cn, jlchen@seu.edu.cn }\\
\end{center}
\bigskip

\indent\indent {\it Let $R$ be a ring. $R$ is called a right countably $\Sigma$-C2 ring if every countable direct sum  copies of $R_{R}$ is a C2 module.   The following are equivalent for a ring $R$: (1) $R$ is a right countably $\Sigma$-C2 ring. (2) The column finite  matrix ring $\mathbb{C}\mathbb{F}\mathbb{M}_{\mathbb{N}}(R)$ is a right C2 (or C3) ring. (3) Every countable direct sum  copies of $R_{R}$ is a C3 module. (4) Every projective right $R$-module is a C2 (or C3) module. (5) $R$ is a right perfect ring and every finite direct sum  copies of $R_{R}$ is a C2 (or C3) module. This shows that right countably $\Sigma$-C2 rings are just the rings whose right finitistic projective dimension r$FPD(R)$=sup\{$Pd_{R}(M)|$ $M$ is a right $R$-module with $Pd_{R}(M)<\infty$\}=0, which were introduced by  Hyman Bass in 1960.}

\bigskip


 \bigskip

 \noindent {\it(2010) Mathematics Subject Classification:} \indent {\rm Primary 16L30; Secondary 16E10.}

\section{\bf  INTRODUCTION}

\indent Throughout this paper rings are associative with
identity.  $\mathbb{N}$ is the set of natural numbers. $\mathbb{M}_{n}(R)$ denotes the ring of $n\times n$ matrices over a ring $R$. Let $\Lambda$ be an infinite set. $\mathbb{C}\mathbb{F}\mathbb{M}_{\Lambda}(R)$ denotes the column finite $card(\Lambda)\times card(\Lambda)$ matrix ring over a ring $R$, where $card(\Lambda)$ is the cardinality of $\Lambda$.   Let $M$ be a right $R$-module and $A$ be a set. $M^{(A)}$ means the direct sum of $card(A)$ copies of $M$.  We use $N\leq_{\oplus}M$ to show that  $N$ is a direct summand of $M$.  And use End($M$) to denote the ring of endomorphisms of $M$.  For a subset $X$ of a ring $R$, the left annihilator of $X$ in $R$ is ${\bf l}(X)=\{r\in R: rx=0$ for all $x\in X\}$. Right annihilators are defined
analogously. \\
\indent Recall that a right $R$-module $M$ is called a \emph{C2} module if every submodule that is isomorphic to a direct summand of $M$ is itself a direct summand of $M$. And $M$ is called a \emph{C3} module if whenever $N\leq_{\oplus}M$ and $K\leq_{\oplus}M$  such that $N\cap K=0$, then $N+K\leq_{\oplus}M$. It is well known that  a C2 module is always a C3 module and the converse is not true. A ring $R$ is called a \emph{right C2} ring if the right $R$-module $R_{R}$ is a C2 module. $R$ is called a  \emph{right finitely $\Sigma$-C2}  ring if  every finite direct sum copies of $R_{R}$ is a C2 module. Right finitely $\Sigma$-C2 rings are also called \emph{strongly right C2} rings (see \cite[Definition 2.2]{CL04}). In this article, we define a ring $R$ to be \emph{right countably $\Sigma$-C2} if $R^{(\mathbb{N})}_{R}$ is a C2 module. The left sides of the above definitions can be defined similarly.\\
 \indent It is proved in Theorem 2.7 that a ring $R$ is a right countably $\Sigma$-C2 ring if and only if $\mathbb{C}\mathbb{F}\mathbb{M}_{\mathbb{N}}(R)$ is a right C2 ring. At the end of the article, it is proved in Theorem 2.13 that a ring $R$ is a right countably $\Sigma$-C2 ring if and only if  $\mathbb{C}\mathbb{F}\mathbb{M}_{\Lambda}(R)$  is a right C2 (or C3) ring for every infinite set $\Lambda$ if and only if  every countable direct sum  copies of $R_{R}$ is a C3 module if and only if every projective right $R$-module is a C2 (or C3) module if and only if  $R$ is a right perfect ring and every finite direct sum  copies of $R_{R}$ is a C2 (or C3) module. This shows  that right countably $\Sigma$-C2 rings are just the rings satisfying  r\emph{FPD}$(R)$=0, which were characterized by Huyman Bass in \cite{B60}.

\section{\bf RESULTS }

\begin{df}\label{def 2.1}
{\rm A ring $R$ is called a right countably $\Sigma$-C2 ring if $R^{(\mathbb{N})}_{R}$
is a C2 module. Left countably $\Sigma$-C2 rings can be defined similarly. A ring $R$ is called a countably $\Sigma$-C2 ring if $R$ is left and right countably $\Sigma$-C2.}
\end{df}

The following Example 2.3 shows that left countably $\Sigma$-C2 rings may not be right countably $\Sigma$-C2.

\begin{lem}\label{lem 2.2}
{\rm(\cite[Lemma 2.3 (2, 3)]{CL04})}
A ring R is right finitely $\Sigma$-C2 if and only if {\bf r}$(I)\neq 0$ for any finitely generated proper left ideal I of R.
\end{lem}

Recall that a ring $R$ is called a \emph{right Kasch} ring if {\bf l}$(I)\neq 0$ for any  proper right ideal $I$ of $R$, or equivalently, if every  maximal right ideal of $R$ is a right annihilator. Left Kasch rings can be defined similarly. By the above lemma, it is clear that every right (left) Kasch ring is  left (right) finitely $\Sigma$-C2.

\begin{exa}\label{exa 2.3}
{\rm  Let $k$ be a division ring, and let $R$ be the ring of matrices
\begin{center}
{$\gamma=\left[
\begin{array}{cccc}
a & 0 & b & c \\
0 & a & 0 & d \\
0 & 0 & a & 0 \\
0 & 0 & 0 & e
\end{array}
\right]$}
\end{center} over $k$. Then $R$ is left countably $\Sigma$-C2 but not right countably $\Sigma$-C2.
}
\end{exa}
\begin{proof}
By \cite[Example 8.29 (6)]{L98}, $R$  is an artinian ring that is right Kasch but not left Kasch. Since  $R$ is artinian and  right Kasch,  $R$ is left perfect and left finitely $\Sigma$-C2.  Then by Theorem 2.13, $R$ is left countably $\Sigma$-C2. Next we show that  $R$ is not right countably $\Sigma$-C2. If $R$ is right countably $\Sigma$-C2,  by Theorem 2.13, $R$ is right finitely $\Sigma$-C2. Thus, by Lemma 2.2, for any finitely generated proper left ideal $I$ of $R$, {\bf r}($I$)$\neq 0$. Since $R$ is artinian, any proper left ideal of $R$ is finitely generated. So $R$ is left Kasch. This is a contradiction.
\end{proof}

 It is well known that any direct summand of a C2 (C3) module is again a C2 (C3) module. So every left countably $\Sigma$-C2 ring is left finitely $\Sigma$-C2. But the converse is not true. For example, let $R$ be a von Neumann regular ring that is not semisimple. Since $R$ is von Neumann regular, every finitely generated one-sided ideal of $R$ is a direct summand of $R$. By Lemma 2.2, it is clear that $R$ is a left and right finitely  $\Sigma$-C2 ring. But according to Theorem 2.13, if $R$ is countably $\Sigma$-C2,  $R$ is a perfect ring. This shows that $R$ is a semisimple ring. It is a contradiction.\\
  \indent The following Example 2.4 is also a left finitely $\Sigma$-C2 ring that is not left countably $\Sigma$-C2. First we recall two definitions.\\
  \indent Let $R$ be a ring and $_{R}W_{R}$  a bimodule. Recall that the \emph{trivial extension} of $W$ by $R$ is the additive group $T(R, W)=R\oplus W$ endowed with the multiplication $(a, w)(a', w')=(aa', aw'+wa')$.\\
  \indent A division ring $D$ is called \emph{existentially closed} over a field $F$ if $D$ is an $F$-algebra and every set of polynomial equations with coefficients in $D$ that is consistent (that is, has a solution in some extension division ring) has a solution in $D$ itself.
\begin{exa}\label{exa 2.4}
 {\rm Let $D$ be any countable, existentially closed division ring over a field $F$, and let $R=D\otimes_{F}$ $F(x)$. Then $S=T(R, D)$ is a left finitely $\Sigma$-C2 ring but not a left countably $\Sigma$-C2 ring. }
\end{exa}
\begin{proof}
By \cite[Example 8.16]{NY03}, $S$ is a right Johns ring which is not right artinian. Recall that a ring $S$ is called a \emph{right Johns} ring if $S$ is right noetherian and \emph{right dual} (every right ideal of $S$ is a right annihilator). Since $S$ is right dual, $S$ is right Kasch. Thus, $S$ is left finitely $\Sigma$-C2. As $S$ is right noetherian but not right artinian, $S$ is not left perfect. By Theorem 2.13, $S$ is not left countably $\Sigma$-C2.
\end{proof}

\begin{lem}\label{lem 2.5}
{\rm (\cite[Theorem 7.14]{NY03})}
Let $M_{R}$ be a module and write E={\rm End(}$M_{R}${\rm)}. Then
\begin{enumerate}
\item If E is a right C2 ring, then $M_{R}$ is a C2 module.\\
\item  The converse in {\rm(1)} holds if {\rm Ker(}$\alpha${\rm)} is generated by M whenever $\alpha\in E$ is such that {\rm\bf{r}}$_{E}${\rm(}$\alpha${\rm)} is a direct summand of $E_{E}$.
\end{enumerate}
\end{lem}

\begin{lem}\label{lem 2.6}
Let $R$ be a ring and $\Lambda$ be an infinite set. Assume $e^{2}=e\in R$. Set M=eR and S=eRe. {\it Then} {\rm End(}$M_{R}^{(\Lambda)}${\rm)}$\cong \mathbb{C}\mathbb{F}\mathbb{M}_{\Lambda}{\rm(}S{\rm)}$.
\end{lem}
\begin{proof}
We prove the case $\Lambda=\mathbb{N}$. The others are similar. To be convenient, we consider $M_{R}^{(\mathbb{N})}$ as the set of all $\mathbb{N}\times 1$ column matrices with finite nonzero entries in $M$. Then for any $\alpha\in M_{R}^{(\mathbb{N})}$ and $A\in \mathbb{C}\mathbb{F}\mathbb{M}_{\mathbb{N}}{\rm(}S{\rm)}$, $A\alpha\in M_{R}^{(\mathbb{N})}$. Now define a map $F$ from $\mathbb{C}\mathbb{F}\mathbb{M}_{\mathbb{N}}{\rm(}S{\rm)}$ to  End($M_{R}^{(\mathbb{N})}$) such that for every $A\in \mathbb{C}\mathbb{F}\mathbb{M}_{\mathbb{N}}{\rm(}S{\rm)}$ and any $\alpha\in M_{R}^{(\mathbb{N})}$, $F(A)(\alpha)=A\alpha$. It is clear that $F$ is a ring homomorphism from $\mathbb{C}\mathbb{F}\mathbb{M}_{\mathbb{\mathbb{N}}}{\rm(}S{\rm)}$ to End($M_{R}^{(\mathbb{N})}$). Next we show that $F$ is an isomorphism. It is easy to see that $F$ is a monomorphism. We only need to show that $F$ is epic. Let $\varepsilon_{i}$ be the element in $M_{R}^{(\mathbb{N})}$  with the $i$th entry equal to $e$ and the others are zero, $\forall i\in \mathbb{N}$. Assume $\varphi\in$ End($M_{R}^{(\mathbb{N})}$). Let $B$=[$\varphi(\varepsilon_{1}),\varphi(\varepsilon_{2}),\ldots,\varphi(\varepsilon_{n})\ldots$]  and $E$=[$\varepsilon_{1}, \varepsilon_{2}, \ldots, \varepsilon_{n}\ldots$] be the matrices with the $i$th column equal to $\varphi(\varepsilon_{i})$ and  $\varepsilon_{i}$, respectively, $i\in\mathbb{N}$. It is clear that $E^{2}=E$ and  $BE\in\mathbb{C}\mathbb{F}\mathbb{M}_{\mathbb{N}}{\rm(}S{\rm)}$. For each $X\in M_{R}^{(\mathbb{N})}$, there exists finite nonzero elements $r_{i}\in eR, i\in\mathbb{N}$, such that $X=\sum_{i=1}^{\infty}\varepsilon_{i}r_{i}$. Let $C$ be the  $\mathbb{N}\times 1$ column matrix with the $i$th entry equal to $r_{i}, i\in\mathbb{N}$. Then $X=EC$. Thus $\varphi(X)=\varphi(\sum_{i=1}^{\infty}\varepsilon_{i}r_{i})=\sum_{i=1}^{\infty}\varphi(\varepsilon_{i})r_{i}=\sum_{i=1}^{\infty}
\varphi(\varepsilon_{i})er_{i}=BEC=BEEC=BEX$. Set $A=BE$. It is clear that $\varphi=F(A)$. Therefore, $F$ is an epimorphism.
\end{proof}
The following Theorem 2.7 shows that $R$ is a right countably $\Sigma$-C2 ring if and only if $\mathbb{C}\mathbb{F}\mathbb{M}_{\mathbb{N}}(R)$ is a right C2 ring. At the end of the article, we will show that $R$ is a right countably $\Sigma$-C2 ring if and only if $\mathbb{C}\mathbb{F}\mathbb{M}_{\Lambda}(R)$ is a right C2 ring for any infinite set $\Lambda$.
\begin{thm}\label{thm 2.7}
Let R be a ring and $\Lambda$ be an infinite set. Then $R_{R}^{(\Lambda)}$ is a C2 module if and only if $\mathbb{C}\mathbb{F}\mathbb{M}_{\Lambda}{\rm(}R{\rm)}$ is a right C2 ring. In particular,  $R$ is a right countably $\Sigma$-C2 ring if and only if $\mathbb{C}\mathbb{F}\mathbb{M}_{\mathbb{N}}(R)$ is a right C2 ring.
\end{thm}
\begin{proof}
By the above lemma,  End($R^{(\Lambda)}_{R}$)$\cong\mathbb{C}\mathbb{F}\mathbb{M}_{\Lambda}(R)$. Since $R^{(\Lambda)}_{R}$ is a generator of right $R$-modules, by Lemma 2.5, $R_{R}^{(\Lambda)}$ is a C2 module if and only if $\mathbb{C}\mathbb{F}\mathbb{M}_{\Lambda}(R)$ is a right C2 ring.
\end{proof}

The following two propositions show that being right countably $\Sigma$-C2 is a Morita invariant.

\begin{prop}\label{prop 2.8}
If a ring $R$ is a right countably $\Sigma$-C2 ring, then $\mathbb{M}_{n}(R)$ is a right countably $\Sigma$-C2 ring for each $n\geq 1$.
\end{prop}
\begin{proof}
For each $n\geq 1$, set $S=\mathbb{M}_{n}(R)$.
Since $R$ is a right countably $\Sigma$-C2 ring, by Theorem 2.7, $\mathbb{C}\mathbb{F}\mathbb{M}_{\mathbb{N}}(R)$ is a right C2 ring. As $\mathbb{C}\mathbb{F}\mathbb{M}_{\mathbb{N}}(S)=\mathbb{C}\mathbb{F}\mathbb{M}_{\mathbb{N}}(R)$, again by Theorem 2.7,  $S$ is a right countably $\Sigma$-C2 ring.
\end{proof}

\begin{prop}\label{prop 2.9}
If $R$ is a right countably $\Sigma$-C2 ring, then eRe is a right countably $\Sigma$-C2 ring, where $e^{2}=e\in R$ and $ReR=R$.
\end{prop}
\begin{proof}
Set $M=eR$ and $S=eRe$. By Lemma 2.6, End($M_{R}^{(\mathbb{N})}$)$\cong \mathbb{C}\mathbb{F}\mathbb{M}_{\mathbb{N}}(S)$. Since $ReR=R$, it is easy to see that $M$ is a generator of right $R$-modules. This implies that $M_{R}^{(\mathbb{N})}$ is also a generator of right $R$-modules. As $R^{(\mathbb{N})}_{R}$ is a C2 module and $M^{(\mathbb{N})}_{R}$ is a direct summand of $R^{(\mathbb{N})}_{R}$, $M^{(\mathbb{N})}_{R}$ is also a C2 module. According to Lemma 2.5, $\mathbb{C}\mathbb{F}\mathbb{M}_{\mathbb{N}}(S)$ is a right C2 ring. Thus, by Theorem 2.7, $S$ is a right countably $\Sigma$-C2 ring.
\end{proof}

The following Theorem 2.13 shows that right countably $\Sigma$-C2  rings are just the rings satisfying r\emph{FPD}($R$)=0, which were characterized by Huyman Bass in \cite{B60}. First we need some lemmas.

\begin{lem}\label{lem 2.10}
{\rm (\cite[Lemma 19.18]{AF92})}
Let R be a ring and  V be a flat right $R$-module and suppose that the sequence $$0\rightarrow K\rightarrow V\rightarrow V^{'}\rightarrow 0$$ is exact. Then $V^{'}$ is flat if and only if for each {\rm(}finitely generated{\rm)} left ideal $I\subseteq$$ _{R}R$, $KI=K\cap VI$.
\end{lem}

The next result is first obtained by Yiqiang Zhou. To be self-contained, we show the proof.

\begin{lem}\label{lem 2.11}{\rm(Yiqiang Zhou)}
Let R be a ring and  M be a right $R$-module. If the direct sum $M\oplus M$ is a C3 module, then $M$ is a C2 module.
\end{lem}
\begin{proof}
Assume $K$ is a submodule of $M$ that is isomorphic to a direct summand $L$ of $M$.  We want to show that $K$ is also a direct summand of $M$. Let $f$ be the isomorphism from $K$ to $L$. Set $K'=\{(x, f(x)): x\in K\}$,  $L'=0\oplus L$ and $M'=M\oplus 0$. Then $K'+M'=M\oplus L$ is a direct summand of $M\oplus M$. Since $K'\cap M'=0$, $K'$ is also a direct summand of $M\oplus M$. It is clear that  $K'\cap L'=0$ and $L'$ is a direct summand of $M\oplus M$. Since $M\oplus M$ is a C3 module. $K'+L'=K\oplus L$ is a direct summand of $M\oplus M$. As $K\oplus 0$ is a direct summand of $K\oplus L$, $K\oplus 0$ is also a direct summand of $M\oplus M$. This shows that $K\oplus 0$ is a direct summand of $M\oplus 0$. It is clear that $K$ is a direct summand of $M$.
\end{proof}

\begin{lem}\label{lem 2.12}
Let R be a ring and $\Lambda$ be an infinite set. Then $R_{R}^{(\Lambda)}$ is a C3 module if and only if $\mathbb{C}\mathbb{F}\mathbb{M}_{\Lambda}{\rm(}R{\rm)}$ is a right C3 ring.
\end{lem}
\begin{proof}
Consider $R_{R}^{(\Lambda)}$ as the set of all $card(\Lambda)\times 1$ column matrices with finite nonzero entries in $R$. It is easy to see that every right ideal $\bar T$ of $\mathbb{C}\mathbb{F}\mathbb{M}_{\Lambda}{\rm(}R{\rm)}$ has the form
$\bar T$=$\{[\alpha~\beta~\gamma~\cdots]|\alpha, \beta, \gamma, \ldots \in T\}$, where $T$ is a submodule of $R_{R}^{(\Lambda)}$. It can be verified  that $\bar T$ is a direct summand of  $\mathbb{C}\mathbb{F}\mathbb{M}_{\Lambda}{\rm(}R{\rm)}$ if and only if $T$ is a direct summand  of $R_{R}^{(\Lambda)}$. And for any two right ideals $\bar T_{1}$ and $\bar T_{2}$ of $\mathbb{C}\mathbb{F}\mathbb{M}_{\Lambda}{\rm(}R{\rm)}$, $\bar T_{1}+\bar T_{2}=\{[\alpha~\beta~\gamma~\cdots]|\alpha, \beta, \gamma, \ldots \in T_{1}+T_{2}\}$ and $\bar T_{1}\cap\bar T_{2}=\{[\alpha~\beta~\gamma~\cdots]|\alpha, \beta, \gamma, \ldots \in T_{1}\cap T_{2}\}$. By the definition of C3 condition, it is easy to see that $R_{R}^{(\Lambda)}$ is a C3 module if and only if $\mathbb{C}\mathbb{F}\mathbb{M}_{\Lambda}{\rm(}R{\rm)}$ is a right C3 ring.

\end{proof}
We recall two dimensions of $R$ ($Pd_{R}(M)$ denotes the projective dimension of a module $M$):\\
r\emph{FPD(R)}=sup\{$Pd_{R}(M)|$ $M$ is a right $R$-module with $Pd_{R}(M)<\infty$\}.\\
r\emph{fPD(R)}=sup\{$Pd_{R}(M)|$ $M$ is a finitely generated right $R$-module with $Pd_{R}(M)<\infty$\}.

\begin{thm}\label{thm 2.13}
 The following are equivalent for a ring  $R$.
 \begin{enumerate}
\item  $R$ is a right countably $\Sigma$-C2 ring.
\item $R^{(\mathbb{N})}_{R}$ is a C3 module.
\item Every free right $R$-module is a C2 module.
\item Every free right $R$-module is a C3 module.
\item Every projective right $R$-module is a C2 module.
\item Every projective right $R$-module is a C3 module.
\item $\mathbb{C}\mathbb{F}\mathbb{M}_{\mathbb{N}}{\rm(}R{\rm)}$ is a right C2 ring.
\item $\mathbb{C}\mathbb{F}\mathbb{M}_{\mathbb{N}}{\rm(}R{\rm)}$ is a right C3 ring.
\item For any infinite set $\Lambda$, $\mathbb{C}\mathbb{F}\mathbb{M}_{\Lambda}{\rm(}R{\rm)}$ is a right C2 ring.
\item For any infinite set $\Lambda$, $\mathbb{C}\mathbb{F}\mathbb{M}_{\Lambda}{\rm(}R{\rm)}$ is a right C3 ring.
\item $R$ is a right perfect ring and right finitely $\Sigma$-C2.
\item $R$ is a right perfect ring and every finite direct sum copies of $R_{R}$ is a C3 module.
\item $R$ is a right perfect ring and rfPD($R$)=0.
\item $R$ is a right perfect ring and every simple right $R$-module is a homomorphic image of an injective module.
\item rFPD($R$)=0.

\end{enumerate}
\end{thm}
\begin{proof}It is clear that (1)$\Rightarrow$(2) and (3)$\Rightarrow$(1). For (2)$\Rightarrow$(1),  since $R^{(\mathbb{N})}_{R}\cong (R^{(\mathbb{N})}_{R}\oplus R^{(\mathbb{N})}_{R})$ and $R^{(\mathbb{N})}_{R}$ is a C3 module, $R^{(\mathbb{N})}_{R}\oplus R^{(\mathbb{N})}_{R}$ is also a C3 module. According to Lemma 2.11, $R^{(\mathbb{N})}_{R}$ is a C2 module. This shows that $R$ is a right countably $\Sigma$-C2 ring. By Theorem 2.7 and Lemma 2.12,  (1)$\Leftrightarrow$(7), (2)$\Leftrightarrow$(8), (3)$\Leftrightarrow$(9) and (4)$\Leftrightarrow$(10). So (1)$\Leftrightarrow$(2)$\Leftrightarrow$(7)$\Leftrightarrow$(8). By Lemma 2.11, (3)$\Leftrightarrow$(4), (5)$\Leftrightarrow$(6) and (11)$\Leftrightarrow$(12).  Since every projective module is a direct summand of a free module and a direct summand of a C2 (C3) module is always a C2 (C3) module, it is easy to see that (3)$\Leftrightarrow$(4)$\Leftrightarrow$(5)$\Leftrightarrow$(6).  Thus (3)$\Leftrightarrow$(4)$\Leftrightarrow$(5)$\Leftrightarrow$(6)$\Leftrightarrow$(9)$\Leftrightarrow$(10).  (11)$\Leftrightarrow$(13)$\Leftrightarrow$(14)$\Leftrightarrow$(15) is obtained by \cite[Theorem 6.3]{B60}. Next we only need to prove (1)$\Rightarrow$(11) and (11)$\Rightarrow$(3).\\
\indent i)(1)$\Rightarrow$(11).  Assume $R$ is a right countably $\Sigma$-C2 ring, then $R$ is right finitely $\Sigma$-C2. Now we show that $R$ is  a right perfect ring. By \cite[Theorem 28.4]{AF92}, we only need to show that $R$ satisfies $DCC$ on principal left ideals of $R$. The following method is owing to Bass. Let $Ra_{1}\supseteq Ra_{2}a_{1}\supseteq \cdots$ be any descending chain of principal left ideals of $R$. Set $F=R^{(\mathbb{N})}_{R}$ with free basis $x_{1}, x_{2},\ldots$ and $G$ be the submodule of $F$ spanned by $y_{i}=x_{i}-x_{i+1}a_{i}, i\in\mathbb{N}$. By \cite[Lemma 28.1]{AF92}, $G$ is free with basis $y_{1}, y_{2},\ldots$. Thus $G\cong F$. Since $R$ is right countably $\Sigma$-C2, $F$ is a C2 module. This implies that $G$ is a direct summand of $F$. Then by \cite[Lemma 28.2]{AF92}, the chain $Ra_{1}\supseteq Ra_{2}a_{1}\supseteq \cdots$ terminates.\\
\indent ii) (11)$\Rightarrow$(3). We only need to show that if $A$ is an infinite set, then $R^{(A)}_{R}$ is a C2 module. Assume $K$ is a submodule of the free module $F=R^{(A)}_{R}$ and $K$ is isomorphic to a direct summand of $F$. In order to show that $K$ is also a direct summand of $F$,  we only need to prove that $F/K$ is  a projective  $R$-module. Since $R$ is right perfect, by \cite[Theorem 28.4]{AF92}, every flat right $R$-module is projective. Thus, we just need to show that $F/K$ is flat. As $R$ is a right perfect ring, $R$ is a semiperfect ring. Then $R$ has a basic set of primitive idempotents $e_{1}, \cdots, e_{m}$. Since $K$ is projective, by \cite[Theorem 27.11]{AF92}, there exist sets $A_{1}, \cdots, A_{m}$ such that $K\cong (e_{1}R)^{(A_{1})}\oplus\cdots\oplus (e_{m}R)^{(A_{m})}$. Set $\lambda$=card($A$). Since $K$ is isomorphic to a direct summand of $F$, $K$ is $\lambda$-generated. So each $(e_{i}R)^{(A_{i})}$ is also $\lambda$-generated, $i=1,2,\ldots, m$. As $\lambda$ is an infinite cardinality, by \cite[Lemma 25.7]{AF92}, card$(A_{i})\leq\lambda$, $i=1,2,\ldots, m$. So card($A_{1}$)+$\cdots$+card($A_{m}$)$\leq m\lambda=\lambda$. Set $L=(e_{1}R)^{(A_{1})}\oplus\cdots\oplus (e_{m}R)^{(A_{m})}$. Then $L$ can be considered as a direct summand of $F$. Let $\mathfrak{A}=\{L_{\alpha}\leq_{\oplus}L: L_{\alpha}\cong(e_{1}R)^{(A_{\alpha_{1}})}\oplus\cdots\oplus (e_{m}R)^{(A_{\alpha_{m}})}$ and card($A_{\alpha_{1}}$)+$\cdots$+card($A_{\alpha_{m}}$) is finite\}. It is clear that $L=\bigcup_{L_{\alpha}\in \mathfrak{A}}L_{\alpha}$ and, for any left ideal $I$ of $R$, $LI=\bigcup_{L_{\alpha}\in \mathfrak{A}}L_{\alpha}I$. Now let $f$ be the isomorphism from $K$ to $L$.  Set $\mathfrak{B}=\{K_{\alpha}=f^{-1} (L_{\alpha}): L_{\alpha}\in \mathfrak{A}\}$. Since $K$ is isomorphic to $L$,  $K=\bigcup_{K_{\alpha}\in \mathfrak{B}}K_{\alpha}$ and, for any left ideal $I$ of $R$, $KI=\bigcup_{K_{\alpha}\in \mathfrak{B}}K_{\alpha}I$.  As $R$ is right finitely $\Sigma$-C2 and  $L_{\alpha}$ is a finitely generated direct summand of $L$ for each $L_{\alpha}\in \mathfrak{A}$, it is easy to verify that $K_{\alpha}$ is a direct summand of $F$ for each $K_{\alpha}\in\mathfrak{B}$.  At last we apply Lemma 2.10 to show that $F/K$ is a flat module. Let $I$ be any left ideal of $R$, by Lemma 2.10, $K_{\alpha}\cap FI=K_{\alpha}I$, $K_{\alpha}\in \mathfrak{B}$. Then $K\cap FI=(\bigcup_{K_{\alpha}\in \mathfrak{B}}K_{\alpha})\cap FI=\bigcup_{K_{\alpha}\in \mathfrak{B}}(K_{\alpha}\cap FI)=\bigcup_{K_{\alpha}\in \mathfrak{B}}K_{\alpha}I=KI$. Thus, by Lemma 2.10, $F/K$ is flat.
\end{proof}

\begin{rem}
 {\rm According to the above theorem, we have\\
 (a) Every right countably $\Sigma$-C2 ring is a right perfect ring. But the converse is not true by  Example 2.3.\\
 (b) It can be directly seen from (15) that being right  $\Sigma$-C2 is a Morita invariant.\\
 (c) Since (11)$\Leftrightarrow$(15) is obtained by \cite[Theorem 6.3]{B60}, the proof of (11)$\Rightarrow$(3) above can be replaced by the following proof of (15)$\Rightarrow$(3). We only need to show that if $A$ is an infinite set, then $R^{(A)}_{R}$ is a C2 module. Assume $K$ is a submodule of the free module $F=R^{(A)}_{R}$ and $K$ is isomorphic to a direct summand of $F$. In order to show that $K$ is also a direct summand of $F$,  we only need to prove that $F/K$ is  a projective  $R$-module. Since $Pd_{R}(F/K)\leq 2$ and r$FPD(R)$=0, $Pd_{R}(F/K)=0$. So $F/K$ is  a projective $R$-module.\\
 (d) The ``right perfect'' in (11) can not be replaced by ``semiperfect''. By \cite[Example 3]{NY97}, there is a commutative, semiperfect, simple-injective, Kasch ring that is not self-injective. Then this ring is a semiperfect and finitely $\Sigma$-C2 ring that is not perfect. If $R$ is perfect, by \cite[Proposition 3]{NY97}, $R$ is a QF ring. This shows that $R$ is self-injective. It is a contradiction.}
\end{rem}

\section*{\bf ACKNOWLEDGEMENTS}
\indent The article was written during the first author's visiting  Center of Ring Theory and Its Applications in Department of
Mathematics, Ohio University. He would like to thank the center for the hospitality. The authors are grateful to Professor Dinh Van Huynh, Professor Sergio R. L$\acute{o}$pez-Permouth and Dr Gangyong Lee for their helpful suggestions.


\begin{thebibliography}{00}

\bibitem{AF92} F. W. Anderson, K. R. Fuller,  \textit{Rings and categories of
modules}, 2nd ed, New York:  Springer-Verlag, 1992.
\bibitem{B60} H. Bass, \textit{Finitistic dimension and a homological generalization of semi-primary rings},
   Trans. Amer. Math. Soc. \textbf{95} (1960), 466--488.
\bibitem{CL04} J. L. Chen, W. X. Li,  \textit{On artiness of right CF rings},
   Comm. Algebra \textbf{32} (2004), 4485--4494.
\bibitem{L98} T. Y.  Lam,  \textit{Lectures on modules and rings},  New York: Springer-Verlag, 1998.
\bibitem{NY97} W. K. Nicholson,  M. F. Yousif, \textit{On perfect simple-injective rings}, Proc. Amer. Math. Soc. \textbf{125} (1997), 979--985.
\bibitem{NY03} W. K. Nicholson,  M. F. Yousif, \textit{Quasi-Frobenius rings}, Cambridge Tracts in Mathematics 158, Cambridge University Press, 2003.

\end{thebibliography}
\end{document}